\documentclass[11pt,reqno]{amsart}
\usepackage{graphicx}
\usepackage{pifont}
\usepackage{asymptote}
\usepackage{asypictureB}
\usepackage{enumitem}
\usepackage{amsbsy}
\setlist[enumerate]{font={\upshape}, label=(\arabic*), leftmargin=*}
\newlist{equivlist}{enumerate}{1}
\setlist[equivlist]{font={\upshape}, label=(\roman*)}
\usepackage{tikz}
\usetikzlibrary{matrix}
\tikzset{ 
	table/.style={
		matrix of nodes,
		nodes={rectangle,text width=1.75em,align=center},
		text depth=1.25ex,
		text height=2.5ex,
		nodes in empty cells
	}
}
\usepackage{mathtools}
\usepackage{cite}
\usepackage{thm-restate}
\usepackage[mathlines]{lineno} 

\usepackage{etoolbox} 

\newcommand*\linenomathpatch[1]{%
	\cspreto{#1}{\linenomath}%
	\cspreto{#1*}{\linenomath}%
	\csappto{end#1}{\endlinenomath}%
	\csappto{end#1*}{\endlinenomath}%
}
\newcommand*\linenomathpatchAMS[1]{%
	\cspreto{#1}{\linenomathAMS}%
	\cspreto{#1*}{\linenomathAMS}%
	\csappto{end#1}{\endlinenomath}%
	\csappto{end#1*}{\endlinenomath}%
}

\expandafter\ifx\linenomath\linenomathWithnumbers
\let\linenomathAMS\linenomathWithnumbers
\patchcmd\linenomathAMS{\advance\postdisplaypenalty\linenopenalty}{}{}{}
\else
\let\linenomathAMS\linenomathNonumbers
\fi
\linenomathpatch{equation}
\linenomathpatchAMS{gather}
\linenomathpatchAMS{multline}
\linenomathpatchAMS{align}
\linenomathpatchAMS{alignat}
\linenomathpatchAMS{flalign}
\usepackage{amsfonts}
\usepackage{amssymb}
\usepackage{scrextend}
\usepackage[pdftex]{hyperref}
\usepackage[margin=0.75in]{geometry}

\makeatletter
\def\thmhead@plain#1#2#3{%
	\thmname{#1}\thmnumber{\@ifnotempty{#1}{ }\@upn{#2}}%
	\thmnote{ {\the\thm@notefont#3}}}
\let\thmhead\thmhead@plain
\makeatother

\newtheorem{theorem}{Theorem}
\newtheorem{question}{Question}

\newtheorem{proposition}[theorem]{Proposition}
\newtheorem{lemma}[theorem]{Lemma}

{%
	\pushQED{\qed}\begin{solution}\upshape}
	{\popQED\end{solution}}

\newcommand{\mac}{\mathcal}
\newcommand{\mab}{\mathbb}
\newcommand{\maf}{\mathbf}

\newcommand{\vep}{\varepsilon}

\newcommand{\xx}{\maf{x}}
\newcommand{\yy}{\maf{y}}
\newcommand{\zz}{\maf{z}}
\newcommand{\vv}{\maf{v}}

\renewcommand{\subset}{\subseteq}

\DeclarePairedDelimiter\abs{\lvert}{\rvert}%
\DeclarePairedDelimiter\norm{\lVert}{\rVert}%
\DeclarePairedDelimiter\ipr{\langle}{\rangle}%
\DeclarePairedDelimiter\floor{\lfloor}{\rfloor}%
\makeatletter
\def\old@comma{,}
\catcode`\,=13
\def,{%
	\ifmmode%
	\old@comma\discretionary{}{}{}%
	\else%
	\old@comma%
	\fi%
}
\makeatother

\begin{document}
	\title{Growing balanced covering sets}
	\author{Tung H. Nguyen}
	\address{Princeton University, Princeton, NJ 08544, USA}
	\email{tunghn@math.princeton.edu}
	\begin{abstract}
		Given a bipartite graph with bipartition $(A,B)$ where $B$ is equipartitioned into $k$ blocks,
		can the vertices in $A$ be picked one by one so that at every step, the picked 
		vertices cover roughly the~same number of vertices in each of these blocks?
		We show that, if 
		each block has cardinality~$m$,
		the vertices in $B$ have the same degree,
		and each vertex in $A$ has at most $cm$ neighbors in every block where $c>0$~is~a small constant,
		then there~is an ordering~$v_1,\ldots,v_n$ 
		of the vertices in $A$ such that for every $j\in\{1,\ldots,n\}$,
		the numbers of vertices with a neighbor in $\{v_1,\ldots,v_j\}$ in every 
		two blocks differ by at most
		$\sqrt{2(k-1)c}\cdot m$.
		This is related to a well-known lemma of Steinitz, 
		and partially answers an unpublished question of Scott~and~Seymour.
	\end{abstract}
	\maketitle
	\section{Introduction}
	For every integer $n\ge1$, let $[n]:=\{1,\ldots,n\}$.
	Let $\mab{N}$ be the set of natural numbers,
	and let $\mab{R}^+$ be the set of nonnegative real numbers.
	The motivation of this note is an unpublished question of Alex Scott and 
	Paul Seymour~\cite{sey2020} on balanced covers of bipartite graphs related to the opening question in the abstract.
	Here, bipartite graphs have no multiple edges.
	\begin{question}
		\label{q:ss}
		Let $k\ge2$ and $m\ge1$ be integers, and let $c\in(0,1/2)$ be a constant independent of $k,m$.
		Consider a bipartite graph
		with bipartition $(A,B)$ where 
		$B$ is partitioned into $k$
		blocks $B_1,\ldots,B_k$
		each of cardinality $m$, the vertices
		in $B$ have the same degree $r\ge1$,
		and each vertex in $A$ has at most $cm$ neighbors in each $B_i$.
		For every $i\in[k]$ and~$S\subset A$,
		let~${N(S,B_i)}$ be the set of vertices in $B_i$ with a neighbor in $S$.
		Does there exist $f\colon\mab{N}\to\mab{R}^+$
		such that there is a chain of sets
		$\emptyset\subsetneq A_1\subsetneq\ldots\subsetneq A_n=A$
		where $n=\abs{A}$
		satisfying
		$\abs{\abs{N(A_j,B_{i_1})}-\abs{N(A_j,B_{i_2})}}
		\le f(k)cm$
		for all $i_1,i_2\in[k]$ and $j\in[n]$?
	\end{question}
	To put Question~\ref{q:ss} into perspective, suppose that we are given 
	disjoint vertex sets $A$ and $B_1,\ldots,B_k$
	with $\abs{B_1}=\ldots=\abs{B_k}=m$
	and each vertex of $A$ only has a small portion of neighbors in each $B_j$.
	In some situations
	(see~\cite[Section 5]{sss2020} for instance), 
	we hope to find a subset $S$ of $A$ such that
	$\abs{N(S,B_1)}$ is roughly $m/2$
	and $\abs{N(S,B_i)}\le m/2$
	for all $i\in\{2,\ldots,k\}$.
	A moment of thought reveals that this can be achieved if 
	the vertices of $A$ can be picked one by one so that at each step $j$ with $A_j$ the set of picked vertices,
	$\abs{N(A_j,B_1)},\ldots,
	\abs{N(A_j,B_k)}$ are roughly the same.
	Indeed, $S$ can be chosen as $A_{j-1}$ where $j\in[n]$ is the smallest index such that there is some $i\in[k]$
	with $\abs{N(A_j,B_i)}>m/2$.
	%
	As a result, for applications it may be desirable to remove the regularity condition on
	$B=B_1\cup\cdots\cup B_k$ in Question~\ref{q:ss}.
	This condition, unfortunately, is in some sense necessary;
	if regularity is changed into almost regularity then $f(2)$ might not even exist,
	as shown by the following proposition.
	\begin{proposition}
		\label{prop:almostreg}
		For every $c,\vep$ with $0<c<1/4$,
		$0<\vep<1$, and $(4c)^{-1}$ an integer,
		there is some $r_0(c,\vep)$ with~the following property.
		For each integer $r\ge r_0(c,\vep)$,
		there exists $m_0(r,c,\vep)$
		such that
		for every integer $m\ge m_0(r,c,\vep)$,
		there is a bipartite graph $G$ with bipartition $(A,B)$ satisfying
		\begin{itemize}
			\item $\abs{A}=c^{-1}r$,
			and $B$ has a partition into two vertex sets $B_1,B_2$
			with $\abs{B_1}=\abs{B_2}=m$,
			
			\item every vertex in $A$ has at most $cm$ neighbors
			in each of $B_1,B_2$,
			
			\item every vertex in $B$ has 
			degree at least $(1-\vep)r$ and at most $r$, and
			
			\item 
			$\abs{\abs{N(S,B_1)}-\abs{N(S,B_2)}}
			\ge \vep m/40$
			for every $S\subset A$
			with $\abs{S}=(4c)^{-1}$.
		\end{itemize}
	\end{proposition}
	\begin{proof}
		[Sketch of proof]
		We make two random bipartite graphs~with bipartitions $(A,B_1)$ and $(A,B_2)$
		where $\abs{A}=c^{-1}r$ and $\abs{B_1}=\abs{B_2}=m$,
		such that for $i=1,2$, every edge between $A$ and $B_i$ is included independently with probability 
		$c(1-(2i-1)\vep/4)$.
		A standard concentration argument shows that there exists $r_0(c,\vep)$
		with the property that for every $r\ge r_0(c,\vep)$, there is some $m_0(r,c,\vep)$ such that
		for each $m\ge m_0(r,c,\vep)$, 
		with positive probability, for every $i=1,2$
		every vertex in $B_i$ has degree at least $(1-i\vep/2)r$ and at most $(1-(i-1)\vep/2)r$,
		and the number of common neighbors of each subset of $A$ of 
		size at~most~$(4c)^{-1}$ in $B_i$ is tightly concentrated around its mean.
		Then an inclusion-exclusion argument finishes the proof.
	\end{proof}
	Still, we believe that Question~\ref{q:ss} is interesting on its own right;
	our result provides a partial answer to it by
	asserting that $f(k)$ can be chosen as $\sqrt{2(k-1)}$ if $c$ is replaced by $\sqrt{c}$.
	\begin{theorem}
		\label{thm:ss}
		Let $k\ge2$ and $m\ge1$ be integers, and let $c\in(0,1/2)$ be a constant independent of~${k,m}$.
		Consider a bipartite graph with bipartition $(A,B)$ where $B$ is partitioned into $k$ blocks
		$B_1,\ldots,B_k$ each of cardinality~$m$,
		the vertices in $B$ have the same degree $r\ge1$,
		and each vertex in $A$ has at most $cm$ neighbors in each $B_i$.
		Then, there is a chain
		$\emptyset\subsetneq A_1\subsetneq\ldots\subsetneq A_n=A$
		where $n=\abs{A}$ such that
		for all $i_1,i_2\in[k]$ and $j\in[n]$,
		$\abs{\abs{N(A_j,B_{i_1})}
			-\abs{N(A_j,B_{i_2})}}
		\le\sqrt{2(k-1)c}\cdot m$.
	\end{theorem}
	In Section~\ref{sec:equiv}, we present an equivalent formulation of Question~\ref{q:ss} which is Question~\ref{q:main},
	and an equivalent statement of Theorem~\ref{thm:ss} which is Theorem~\ref{thm:main}.
	In Section~\ref{sec:main},
	we prove Theorem~\ref{thm:main}.
	Section~\ref{sec:additional} is a brief discussion about a variant of Question~\ref{q:main}.
	\section{An equivalent formulation of Question~\ref{q:ss}}
	\label{sec:equiv}
	In this section, we introduce an equivalent formulation of 
	Question~\ref{q:ss} which is more convenient to work with.
	For a finite set $S$ and an integer $r\ge1$, let $S^{(r)}$ be the family of all subsets of 
	cardinality~$r$~of~$S$; we~identify~$S^{(1)}$ with $S$.
	For integers $n,r$ with $n\ge r\ge1$,
	a \emph{weighted $r$-uniform hypergraph on $[n]$}
	is a function $w\colon [n]^{(r)}\to\mab{R}^+$
	satisfying ${\sum_{R\in[n]^{(r)}}w(R)=1}$.
	For~every $S\subset[n]$, let $w^*(S):=\sum_{R\in S^{(r)}}w(R)$.
	Thus $w^*(S)=0$ when $\abs{S}<r$.
	Now, Question~\ref{q:ss} can be rephrased as~follows.
	\begin{question}
		\label{q:main}
		Let $k\ge2$ be an integer and ${c\in(0,1/2)}$.
		Let $w_1,\ldots,w_k$ be weighted $r$-uniform hypergraphs on $[n]$
		satisfying~${w_i^*([n]\setminus\{j\})\ge1-c}$
		for all $i\in[k]$ and $j\in[n]$,
		where $n,r$ are integers with ${n\ge r\ge1}$.
		Does there exist $f\colon\mab{N}\to\mab{R}^+$ such that there is a chain
		$\emptyset\subsetneq S_1\subsetneq
		\ldots\subsetneq S_n=[n]$
		satisfying
		$\abs{w_{i_1}^*(S_j)-w_{i_2}^*(S_j)}\le f(k)c$
		for all $i_1,i_2\in[k]$ and $j\in[n]$?
	\end{question}
	\begin{proof}
		[Proof of the equivalence of Questions~\ref{q:ss} and~\ref{q:main}]
		First, assume that Question~\ref{q:main} has a positive answer with some $f\colon\mab{N}\to\mab{R}^+$.
		To see that $f$ answers Question~\ref{q:ss} in the positive,
		we identify $A$ with $[n]$, and let 
		$w_i(R):=1-\abs{N(A\setminus R,B_i)}/m$ for all $i\in[k]$ and $R\in A^{(r)}$;
		then
		$w_i^*(S)=1-\abs{N(A\setminus S,B_i)}/m$
		for all $i\in[k]$ and $S\subset A=[n]$,
		in particular
		$w_i^*([n]\setminus\{j\})=1-\abs{N(j,B_i)}/m\ge 1-c$
		for all~$j\in[n]$.
		Let $S_n:=[n]$ and $S_j:=[n]\setminus A_{n-j}$ for every $j\in[n-1]$,
		then $\emptyset\subsetneq S_1\subsetneq\ldots\subsetneq S_n=[n]$,
		hence 
		\[\abs{\abs{N(A_j,B_{i_1})}-\abs{N(A_j,B_{i_2})}}
		=m\abs{w_{i_1}^*(S_{n-j})-w_{i_2}^*(S_{n-j})}
		\le f(k)cm
		\quad
		\text{for all $i_1,i_2\in[k]$ and $j\in[n-1]$}.\]
		Moreover, $\abs{N(A_n,B_i)}=\abs{N(A,B_i)}
		=\abs{B_i}=m$ for all $i\in[k]$
		as $r\ge1$.
		Therefore $f$ answers Question~\ref{q:ss} in the positive.	
		
		Now, assume that Question~\ref{q:ss}
		has a positive answer with some $f\colon\mab{N}\to\mab{R}^+$.
		To see that $f$ answers Question~\ref{q:main} in the positive,
		observe that it suffices to consider when each $w_i$ assumes rational values,
		in which case there exists $m$ such that $m\cdot w_i(R)$ is an integer for all $i\in[k]$ and~$R\in [n]^{(r)}$.
		We then let $A:=[n]$, and let each $B_i$ have cardinality $m$ and have precisely 
		$m\cdot w_i(R)$ vertices each having neighborhood $R$ for every $R\in[n]^{(r)}$.
		Since $\sum_{R\in[n]^{(r)}}w_i(R)=1$,
		every vertex in $B_i$ has degree~$r$
		and
		$\abs{N(A\setminus S,B_i)}/m
		=1-w_i^*(S)$
		for all $i\in[k]$ and $S\subset [n]=A$,
		thus $\abs{N(j,B_i)}/m
		=1-w_i^*([n]\setminus\{j\})\le c$
		for all $j\in[n]$.
		Let $A_n:=A$ and $A_j:=A\setminus S_{n-j}$ for every $j\in[n-1]$, then
		$\emptyset\subsetneq A_1\subsetneq\ldots\subsetneq A_n=A$,
		so
		\[\abs{w_{i_1}^*(S_j)-w_{i_2}^*(S_j)}
		=\abs{\abs{N(A_{n-j},B_{i_1})}-
			\abs{N(A_{n-j},N_{i_2})}}/m
		\le f(k)c
		\quad
		\text{for all $i_1,i_2\in[k]$ and $j\in[n-1]$}.\]
		Moreover, $w_i^*(S_n)=w_i^*([n])=1$ for all $i\in[k]$.
		Therefore $f$ answers  Question~\ref{q:main}
		in the positive.
	\end{proof}
	The above proof also shows that Theorem~\ref{thm:ss} is equivalent to a result that partially answers Question~\ref{q:main};
	we shall prove this result in Section~\ref{sec:main}.
	\begin{theorem}
		\label{thm:main}
		Let $k\ge2$ be an integer and $c\in(0,1/2)$.
		Let $w_1,\ldots,w_k$ be weighted $r$-uniform hypergraphs on~$[n]$ for some $n\ge r\ge1$.
		If $w_i^*([n]\setminus\{j\})\ge1-c$ for every $i\in[k]$ and $j\in[n]$,
		then there is a chain~of~sets $\emptyset\subsetneq S_1\subsetneq\ldots\subsetneq S_n=[n]$
		such that
		$\abs{w_{i_1}^*(S_j)-w_{i_2}^*(S_j)}\le\sqrt{2(k-1)c}$
		for all $i_1,i_2\in[k]$ and $j\in[n]$.
	\end{theorem}
	
	We would like to make three remarks.
	First, in Questions~\ref{q:ss} and~\ref{q:main}, $f(2)$ can be chosen to be $1$, 
	but it is still open whether $f(3)$ exists.
	It would also be helpful to know whether $f$ exists when $r=2$.
	
	Second, if one restricts Question~\ref{q:main} to the case $r=1$ only,
	then one can choose $f(k)=2(k-1)$; 
	this follows from a well-known lemma of Steinitz
	(see~\cite{bar2008}
	for its history and related~results),
	which~we~state~here.
	In what follows, $\norm{.}_{\infty}$ denotes the $\infty$-norm.
	\begin{theorem}
		\label{thm:steinitz}
		Let $k\ge1$ be an integer.
		Then for every finite subset~$V$~of~$\mathbb{R}^k$~with
		$\norm{\vv}_{\infty}\le1$ for all~$\vv\in V$
		and $\sum_{\vv\in V}\vv=\maf{0}$,
		there is an ordering
		$\vv_1,\ldots,\vv_n$ of the vectors in $V$ with
		$\norm{\vv_1+\cdots+\vv_j}_{\infty}\le k$
		for~all~$j\in[n]$.
	\end{theorem}
	To see how $f(k)$ can be chosen as $2(k-1)$ when $r=1$ in Question~\ref{q:main},
	for each $j\in[n]$
	let $\vv_j$ be the vector in $\mab{R}^{k-1}$ whose $i$-th component is
	$w_{i}(j)-w_{k}(j)$
	for all $i\in[k-1]$,
	then $\norm{\vv_j}_{\infty}\le c$. 
	We apply Theorem~\ref{thm:steinitz}
	to
	$V=\{\vv_1,\ldots,\vv_n\}$, noting that
	$w_i^*(S)=\sum_{j\in S}w_i(j)$
	and
	$\abs{w_{i_1}^*(S)-w_{i_2}^*(S)}
	\le\abs{w_{i_1}^*(S)-w_k^*(S)}
	+\abs{w_{i_2}^*(S)-w_k^*(S)}$
	for all $i,i_1,i_2\in[k]$ and $S\subset[n]$.
	Theorem~\ref{thm:main} does not imply Theorem~\ref{thm:steinitz} as far as we know.
	
	Third, every function $f$ answering Question~\ref{q:main}
	in the positive, if exists, satisfies $f(k)=\Omega(\sqrt{k})$.
	To see this, for integer $k\ge1$
	let $g(k)$ be the smallest constant such that for each finite $V\subset\mab{R}^k$ with~$\norm{\vv}_{\infty}\le1$ for all $\vv\in V$ and $\sum_{\vv\in V}\vv=\maf{0}$,
	there is an ordering $\vv_1,\ldots,\vv_n$ of the vectors in $V$ with $\norm{\vv_1+\cdots+\vv_j}_{\infty}\le g(k)$ for all $j\in[n]$;
	then $g(k)\le k$ by Theorem~\ref{thm:steinitz}.
	\begin{proposition}
		$f(k)\ge g(k-1)$ for all $k\ge2$.
	\end{proposition}
	\begin{proof}
		%
		Fix $\theta\in(0,1)$, then by the definition of $g(k-1)$,
		there exist an integer $n\ge1$ and $\vv_1,\ldots,\vv_n\in\mab{R}^{k-1}$ with $\norm{\vv_j}_{\infty}\le\theta c$ for all $j\in[n]$ and $\sum_{j=1}^n\vv_j=\maf{0}$, 
		such that for every chain $\emptyset\subsetneq S_1\subsetneq\ldots\subsetneq S_n=[n]$,
		there is some $j_0\in[n]$ with $\norm{\sum_{j\in S_{j_0}}\vv_j}_{\infty}>(\theta g(k-1))(\theta c)=\theta^2g(k-1)c$.
		By adding zero vectors if necessary, we may assume $n\ge((1-\theta)c)^{-1}$.
		Define $w_1,\ldots,w_k\colon[n]
		\to\mab{R}^+$
		by the rule that for every $j\in[n]$,
		$w_k(j):=1/n$ and
		$w_i(j):=w_k(j)+v_{ji}$
		where $v_{ji}$ is the $i$-th coordinate of $\vv_j$
		for all $i\in[k-1]$.
		Then $w_1,\ldots,w_k$ are weighted $1$-uniform hypergraphs on $[n]$;
		and as $n\ge((1-\theta)c)^{-1}$,
		$w_i(j)=w_k(j)+v_{ji}\le(1-\theta)c+\theta c=c$ for every $i\in[k-1]$ and $j\in[n]$.
		But~${\vv_j=(w_1(j)-w_k(j),\ldots,w_{k-1}(j)-w_k(j))}$
		for all $j\in[n]$,
		so for every chain
		$\emptyset\subsetneq S_1\subsetneq\ldots\subsetneq S_n=[n]$,
		there exists $j_0\in[n]$ with
		$\max_{i\in[k-1]}\abs{w_{i}^*(S_{j_0})-w_k^*(S_{j_0})}>\theta^2g(k-1)c$.
		
		Thus, $f(k)>\theta^2g(k-1)$ for all $\theta\in(0,1)$ hence $f(k)\ge g(k-1)$. This completes the proof.
	\end{proof}
	By the construction in~\cite[Section 3]{bar2008}, $g(k)\ge\sqrt{k}/2$ whenever $k$ is the order of some Hadamard matrix.
	For $k\ge1$ in general, by Sylvester's construction which yields a Hadamard matrix of order equal to an arbitrary power of two, 
	in particular of order $2^{\floor{\log_2k}}$,
	it is not hard to see that $g(k)\ge2^{\floor{\log_2k}/2}/2\ge\sqrt{2k}/4$.
	Hence $f(k)\ge g(k-1)\ge\sqrt{2(k-1)}/4$ for all $k\ge2$
	so $f(k)=\Omega(\sqrt{k})$.
	We note that showing the existence of $f(k)$ for $k\ge2$ is equivalent to proving that $f(k)$ is bounded from above by a function of $g(k)$;
	and there is also a long-standing conjecture (according to~\cite{bar2008}) that $g(k)=O(\sqrt{k})$.
	\section{Proof of Theorem~\ref{thm:main}}
	\label{sec:main}
	In this section, we prove Theorem~\ref{thm:main}.
	For every $i\in[k]$, $S\subset[n]$, and $j\in S$, let
	\[\delta_i(j,S):=w_i^*(S)-w_i^*(S\setminus\{j\})
	=\sum_{R\in S^{(r)},j\in R}w_i(R).\]
	Then $0\le\delta_i(j,S)\le\delta_i(j,[n])\le c$; and moreover, 
	for all $i\in[k]$ and $S\subset[n]$,
	\begin{equation}
		\label{eq:2}
		\sum_{j\in S}\delta_i(j,S)=r\cdot w_i^*(S).
	\end{equation}
	
	For $S\subset[n]$, let the \emph{unbalance} of $S$ be the quantity
	$\max_{i_1,i_2\in[k]}
	\abs{w_{i_1}^*(S)-w_{i_2}^*(S)}$.
	To prove Theorem~\ref{thm:main},
	we shall build the desired chain
	$\emptyset\subsetneq S_1\subsetneq \ldots\subsetneq S_n=[n]$ in reverse
	so that $S_j$ has unbalance at most $\sqrt{2(k-1)c}$
	for all $j\in[n]$.
	So, given a nonempty subset $S$ of $[n]$ whose unbalance is reasonably small,
	it might be helpful to see how we can  remove some $j\in S$ while 
	maintaining reasonably small unbalance.
	When $k=2$, by~\eqref{eq:2}
	there is some $j\in S$ for which 
	$\delta_1(j,S)-\delta_2(j,S)$ has the same sign
	as $w_1^*(S)-w_2^*(S)$ has,
	and thus $S\setminus\{j\}$ has unbalance at most $c$ if
	$S$ has unbalance~at~most $c$.
	(This explains why one can choose $f(2)=1$ in Questions~\ref{q:ss} and~\ref{q:main}.)
	When $k\ge3$, note that
	\begin{equation}
		\label{eq:1}
		(w_{i_1}^*(S)-w_{i_2}^*(S))^2
		\le 2[(w_{i_1}^*(S)-w_{k}^*(S))^2
		+(w_{i_2}^*(S)-w_k^*(S))^2]
		\le 2\sum_{i=1}^{k-1}
		(w_i^*(S)-w_k^*(S))^2,
	\end{equation}
	which leads us to pay attention to how the quantity
	\[\norm{\boldsymbol{\varphi}(S)}^2
	=\sum_{i=1}^{k-1}(w_i^*(S)-w_k^*(S))^2
	\]
	changes when we delete an element from $S$.
	Here, $\boldsymbol{\varphi}(S)$ is the vector in $\mab{R}^{k-1}$ whose $i$-th component is
	$w_i^*(S)-w_k^*(S)$,
	and $\norm{.}$ is the Euclidean norm in $\mab{R}^{k-1}$.
	The following lemma is motivated by this idea, showing that
	if $\norm{\boldsymbol{\varphi}(S)}$ is reasonably small and $j$ is a uniformly random element of $S$,
	then $\norm{\boldsymbol{\varphi}(S\setminus\{j\})}$
	is reasonably small in expectation.
	\begin{lemma}
		\label{lem:main}
		If $S$ is a nonempty subset of $[n]$, then
		\[\frac{1}{\abs{S}}
		\sum_{j\in S}
		\norm{\boldsymbol{\varphi}(S\setminus\{j\})}^2
		\le \norm{\boldsymbol{\varphi}(S)}^2
		-\frac{2r}{\abs{S}}
		(\norm{\boldsymbol{\varphi}(S)}^2-(k-1)c).\]
	\end{lemma}
	\begin{proof}
		For every $j\in S$, let
		\[\xx_j:=\boldsymbol{\varphi}(S)-\boldsymbol{\varphi}(S\setminus\{j\})
		=(\delta_{1}(j,S)
		-\delta_{k}(j,S),
		\ldots,
		\delta_{k-1}(j,S)
		-\delta_k(j,S)),\]
		then~\eqref{eq:2} yields
		\[\sum_{j\in S}\xx_j
		=r(w_1^*(S)-w_k^*(S),\ldots,w_{k-1}^*(S)-w_k^*(S))
		=r\boldsymbol{\varphi}(S).\]
		Write $\ipr{.,.}$ for the standard inner product in $\mab{R}^{k-1}$.
		It follows that
		\begin{align*}
			\sum_{j\in S}
			\norm{\boldsymbol{\varphi}(S\setminus\{j\})}^2
			=\abs{S}\cdot\norm{\boldsymbol{\varphi}(S)}^2
			-2\sum_{j\in S}\ipr{\boldsymbol{\varphi}(S),\xx_j}
			+\sum_{j\in S}\norm{\xx_j}^2
			=\abs{S}\cdot\norm{\boldsymbol{\varphi}(S)}^2
			-2r\norm{\boldsymbol{\varphi}(S)}^2
			+\sum_{j\in S}\norm{\xx_j}^2.
		\end{align*}
		To conclude the proof of the lemma, 
		it suffices to show that
		$\sum_{j\in S}\norm{\xx_j}^2
		\le 2r(k-1)c$.
		To this end,
		let $x_{ij}:=\delta_{i}(j,S)-\delta_{k}(j,S)$
		for every $i\in[k-1]$ and $j\in S$.
		Observe that, for each $i\in[k-1]$,
		\[\sum_{j\in S}x_{ij}^2
		\le \sum_{j\in S}
		(\delta_{i}(j,S)^2+\delta_{k}(j,S)^2)
		\le c\sum_{j\in S}
		(\delta_{i}(j,S)+\delta_{k}(j,S))
		\overset{\eqref{eq:2}}{=}
		rc(w_i(S)+w_k(S))
		\le 2rc.\]
		Therefore
		\(\sum_{j\in S}\norm{\xx_j}^2
		=\sum_{j\in S}\sum_{i=1}^{k-1}x_{ij}^2
		=\sum_{i=1}^{k-1}\sum_{j\in S}x_{ij}^2
		\le 2r(k-1)c,\)
		as claimed.
	\end{proof}
	We are now ready to finish the proof of Theorem~\ref{thm:main}.
	\begin{proof}
		[Proof of Theorem~\ref{thm:main}]
		If $n\le2r$, we order the elements of $[n]$ arbitrarily, obtaining the chain
		$\emptyset\subsetneq S_1\subsetneq\ldots\subsetneq S_n=[n]$.
		For every $j\in[n]$, $w_i^*(S_j)\le2c$ for each $i\in[k]$ by~\eqref{eq:2},
		so~$S_j$ has unbalance at most
		$2c\le\sqrt{2(k-1)c}$ where we assumed $k\ge3$.
		If $n>2r$, we first construct 
		$ S_{2r}\subsetneq \ldots\subsetneq S_{n-1}\subsetneq S_n=[n]$ by backward induction
		where $\abs{S_j}=j$ and $\norm{\boldsymbol{\varphi}(S_j)}^2\le (k-1)c$
		for all $j$ with $2r< j\le n$;
		by~\eqref{eq:1},
		it follows that $S_j$ has unbalance at most $ \sqrt{2(k-1)c}$
		for such $j$.
		Initially $\norm{\boldsymbol{\varphi}(S_n)}
		=0$
		as $S_n=[n]$.
		For $2r<j\le n$, assume that we have constructed $S_{j},\ldots,S_{n-1},S_n$.
		By Lemma~\ref{lem:main}~with $S=S_{j}$,
		there exists $j_0\in S_{j}$ with
		$\norm{\boldsymbol{\varphi}(S_{j}\setminus\{j_0\})}^2
		\le (k-1)c$,
		and we~let~$S_{j-1}:=S_{j}\setminus\{j_0\}$.
		This finishes the construction of $S_{2r},\ldots,S_{n-1},S_n$.
		We then order the elements of $S_{2r}$ arbitrarily,
		obtaining the chain
		$\emptyset\subsetneq S_1\subsetneq\ldots\subsetneq S_{2r-1}\subsetneq S_{2r}$,
		and arguing similarly as in the case $n\le2r$.
		This completes the construction
		and the proof of Theorem~\ref{thm:main}.
	\end{proof}
	\section{Additional remarks}
	\label{sec:additional}
	We would like to discuss a variant
	of Question~\ref{q:main}
	that corresponds to a result of B\'{a}r\'{a}ny and Grinberg~\cite[Theorem 4.1]{bar2008}.
	For integer $n\ge1$,
	let $\mac{F}_n$ be the family of all subsets of $[n]$.
	A \emph{weighted hypergraph on $[n]$} is a function $w\colon \mac{F}_n\to\mab{R}^+$
	with $w(\emptyset)=0$
	and $\sum_{X\subset[n]}w(X)=1$;
	a subset $X$ of $[n]$ is called an \emph{edge} of $w$ if $w(X)>0$.
	Thus, if the edges of $w$ have the same cardinality $r$ for some $r\in[n]$,
	then $w$ can be viewed as a weighted $r$-uniform hypergraph on $[n]$.
	For every $S\subset[n]$,
	let $w^*(S):=\sum_{X\subset S}w(X)$.
	\begin{question}
		\label{q:partition}
		Let $k\ge1$ be an integer, and let $c>0$.
		Let $w_1,\ldots,w_k$ be weighted hypergraphs on $[n]$
		with $w_i^*([n]\setminus\{j\})\ge 1-c$ for all $i\in[k]$ and $j\in[n]$,
		for some integer $n\ge1$.
		Does there exist $f\colon\mab{N}\to\mab{R}^+$
		such that there is a partition $[n]=S\cup T$ with 
		$\abs{w_i^*(S)-w_i^*(T)}\le f(k)c$
		for all $i\in[k]$?
	\end{question}
	D\"{o}m\"{o}t\"{o}r P\'{a}lv\"{o}lgyi~\cite{palvolgyi2021}
	observed that $f(k)=2k$ answers Question~\ref{q:partition} in the positive.
	With his permission, we present his argument here.
	\begin{proposition}
		Let $k\ge1$ be an integer, and let $c>0$. Let $w_1,\ldots,w_k$ be weighted hypergraphs on $[n]$~with 
		$w_i^*([n]\setminus\{j\})\ge1-c$ for all $i\in[k]$ and $j\in[n]$,
		for some $n\ge1$.
		Then there is a partition $[n]=S\cup T$ with $\abs{w_i^*(S)-w_i^*(T)}\le2kc$
		for all $i\in[k]$.
	\end{proposition}
	\begin{proof}
		Let 
		$\mac{U}:=\{(A,B):A,B\subset[n],A\cap B=\emptyset,A\cup B\ne\emptyset\}$.
		The proof makes use of
		the octahedral Tucker lemma~\cite{palvolgyi2009,ziegler2005}, which we state here.
		\begin{lemma}
			\label{lem:octa}
			For an integer $n\ge1$,
			if there exists a function 
			$\lambda\colon\mac{U}\to\{-n+1,-n+2,\ldots,-1,1,2,\ldots,n-1\}$
			such that $\lambda(S,T)=-\lambda(T,S)$
			for all $(S,T)\in\mac{U}$,
			then there exist $(S_1,T_1),(S_2,T_2)\in\mac{U}$
			with $S_1\subset S_2$ and $T_1\subset T_2$ such that
			$\lambda(S_1,T_1)=-\lambda(S_2,T_2)$.
		\end{lemma}
		Now, we may assume that $n\ge 2k$.
		For every $(S,T)\in\mac{U}$, define $\lambda(S,T)$ by the rule that
		\begin{itemize}[leftmargin=2.5em]
			\item if $\abs{S}+\abs{T}<n-k$,
			then let $\lambda(S,T)$ be either $k+\abs{S}+\abs{T}$ or $-(k+\abs{S}+\abs{T})$
			(sign chosen arbitrarily so that
			$\lambda(S,T)=-\lambda(T,S)$ in this case);
			
			\item if $\abs{S}+\abs{T}\ge n-k$
			and $w_{i}^*(S)-w_{i}^*([n]\setminus S)>0$ for some $i\in[k]$,
			then let $\lambda(S,T)$ be such an $i$; and
			
			\item if $\abs{S}+\abs{T}\ge n-k$
			and $w_{i}^*(T)-w_{i}^*([n]\setminus T)>0$ for some $i\in[k]$,
			then let $-\lambda(S,T)$ be such an $i$.
		\end{itemize} 
		
		Now, if every $\lambda(S,T)$ were defined, then it would not be difficult to verify that $\lambda(S,T)=-\lambda(T,S)$ for all $(S,T)\in\mac{U}$ and that
		there would not exist
		a complementary containment pair relative to $\lambda$, a contradiction by Lemma~\ref{lem:octa}.
		So there exists $(S,T)\in\mac{U}$ with $\abs{S}+\abs{T}\ge n-k$
		such that $w_i^*(S)-w_i^*([n]\setminus S)\le 0$ and
		$w_i^*(T)-w_i^*([n]\setminus T)\le 0$
		for all $i\in[k]$.
		Put $Z:=[n]\setminus(S\cup T)$,
		then $\abs{Z}\le k$. 
		For every $i\in[k]$, the condition
		$w_i^*([n]\setminus\{j\})\ge1-c$ for all $j\in[n]$ yields
		$w_i^*(S\cup Z)\le w_i^*(S)+c\abs{Z}$
		and 
		$w_i^*(T\cup Z)\le w_i^*(T)+c\abs{Z}$
		which together imply
		$w_i^*(S\cup Z)-w_i^*(T)
		\le w_i^*(S)+c\abs{Z}-w_i^*([n]\setminus S)+c\abs{Z}\le 2kc$.
		Similarly,
		for all $i\in[k]$,
		$w_i^*(T\cup Z)-w_i^*(S)\le 2kc$
		thus $w_i^*(T)-w_i^*(S\cup Z)\le 2kc$.
		Hence $[n]=(S\cup Z)\cup T$ is a desired partition.
	\end{proof}
	
	It is unknown whether
	$f(k)$ can be chosen to be of order $o(k)$ in Question~\ref{q:partition},
	but if one requires~${w_1,\ldots,w_k}$ to have edges of cardinality at most two,
	then one can choose $f(k)=6\sqrt{k}$;
	we present a proof of this result.
	
	\begin{proposition}
		Let $k\ge1$ be an integer, and let $c>0$. Let $w_1,\ldots,w_k$ be weighted hypergraphs on $[n]$~with 
		$w_i^*([n]\setminus\{j\})\ge1-c$ for all $i\in[k]$ and $j\in[n]$,
		for some $n\ge1$.
		If the edges of $w_i$ have cardinality at most two for each $i\in[k]$,
		then there is a partition $[n]=S\cup T$ with $\abs{w_i^*(S)-w_i^*(T)}\le6\sqrt{k}\cdot c$
		for all $i\in[k]$.
	\end{proposition}
	\begin{proof}
		We follow the idea illustrated in~\cite[Section 4]{bar2008}.
		For every $j\in[n]$, let
		$\xx_j:=(x_{1j},\ldots,x_{kj})\in\mab{R}^k$ where
		\[x_{ij}=w_i(j)+\frac{1}{2}\sum_{j'\in[n]\setminus\{j\}}w_i(\{j,j'\})
		\quad
		\text{for every $i\in[k]$},\]
		then $0\le x_{ij}\le w_i^*([n])-w_i^*([n]\setminus\{j\})\le c$
		for all $i\in[k]$;
		thus $\norm{\xx_j}_{\infty}\le c$.
		Define the convex polytope
		\[\mac{P}:=\{(a_1,\ldots,a_n)\in\mab{R}^n:
		-1\le a_1,\ldots,a_n\le 1
		\text{ and }
		a_1\xx_1+\cdots+a_n\xx_n=\maf{0}\}.\]
		Observe that $\mac{P}\ne\emptyset$
		since $\maf{0}\in\mac{P}$;
		thus let $(a_1,\ldots,a_n)\ne\maf{0}$ be an extreme point of $\mac{P}$.
		The system of linear equations defining $\mac{P}$ has $k$ equations,
		so $\abs{J}\le k$ where $J=\{j\in[n]:-1<a_j<1\}$.
		We may assume $J=[p]$ for some $p\in[n]$.
		Put $\yy:=\sum_{j=p+1}^na_j\xx_j
		=-\sum_{j=1}^pa_j\xx_j$;
		then by Spencer's six standard deviations theorem~\cite{spen1985},
		there exists $\zz=\sum_{j=1}^{p}b_j\xx_j$ with $b_1,\ldots,b_p\in\{-1,1\}$
		such that ${\norm{\zz-\yy}_{\infty}
			\le6\sqrt{p}\cdot c\le 6\sqrt{k}\cdot c}$.
		Let $\xi_j:=b_j$ for each $j\in[p]$
		and let $\xi_j:=a_j$ for each $j\in[n]\setminus[p]$,
		then $\sum_{j=1}^n\xi_j\xx_j=\zz-\yy$
		has $\infty$-norm at most $6\sqrt{k}\cdot c$.
		Put $S:=\{j\in[n]:\xi_j=1\}$
		and $T:=\{j\in[n]:\xi_j=-1\}$;
		then for every $i\in[k]$,
		\[\sum_{j\in S}x_{ij}
		=\sum_{j\in S}\left(w_i(j)
		+\frac{1}{2}\sum_{j'\in[n]\setminus\{j\}}w_i(\{j,j'\})\right)
		=w_i^*(S)+\frac{1}{2}\sum_{j\in S,j'\in T}w_i(\{j,j'\}),\]
		and similarly,
		\[\sum_{j\in T}x_{ij}
		=w_i^*(T)+\frac{1}{2}\sum_{j\in T,j'\in S}w_i(\{j,j'\}).\]
		It follows that
		\[\abs{w_i^*(S)-w_i^*(T)}
		=\abs*{\sum_{j\in S}x_{ij}-\sum_{j\in T}x_{ij}}
		=\abs{\xi_1x_{i1}+\cdots+\xi_nx_{in}}
		\le 6\sqrt{k}\cdot c\]
		for all $i\in[k]$.
		This completes the proof.
	\end{proof}
	\subsection*{Acknowledgements}
	The author would like to thank Paul Seymour for introducing him to Question~\ref{q:ss},
	for encouragement,
	and for helpful comments.
	He would also like to thank
	the anonymous referees for valuable suggestions.
	
\end{document}